\documentclass[a4paper,11pt]{amsart}

\usepackage{amsmath,amssymb,amsfonts,latexsym}

\newtheorem{theorem}{Theorem}[section]

\newtheorem{corollary}[theorem]{Corollary}

\begin{document}

\title[Generalized fractional integration of $k$-Bessel function]{Generalized fractional integration of $k$-Bessel function}

\author{G.  Rahman}
\address{Department of Mathematics, International Islamic  University, City Islamabad, Country Pakistan.}
\email{gauhar55uom@gmail.com}

\author{K.S. Nisar$^{\ast}$}
\address{Department of Mathematics,College of Arts and Science-Wadi Al dawaser, Prince Sattam bin Abdulaziz University, Saudi Arabia.}
\email{ksnisar1@gmail.com}

\author{S. Mubeen}
\address{Department of Mathematics, University of Sargodha,   City Sargodha, Country Pakistan.}
\email{smjhanda@gmail.com}

\author{M. Arshad}
\address{Department of Mathematics, International Islamic  University, City Islamabad, Country Pakistan.}
\email{marshad\_zia@yahoo.com}

\begin{abstract}
In this present paper our aim is to deal with two integral transforms  which involving the Gauss hypergeometric function as its kernels.   We prove some compositions formulas for such a generalized  fractional integrals with $k$-Bessel function. The results are established in terms of generalized Wright type hypergeometric  function and generalized hypergeometric series. Also,  the authors presented some corresponding assertions for Riemann–Liouville  and Erd\'{e}lyi–Kober fractional integral transforms.
\vspace{2mm}


\vspace{2mm}


\end{abstract}

\thanks{*corresponding author}


\maketitle


\section {Introduction and Preliminaries}


The Gauss hypergeometric function is defined as:
\begin{equation}\label{1}
_2F_1(a,b;c;z)=\sum\limits_{n=0}^{\infty}\frac{(a)_n(b)_n}{(c)_n}\frac{z^n}{n!},
\end{equation}
where $a,b,c \in\mathbb{C}$, $c\neq0,-1,-2,\cdots$ and $(\lambda)_n$ is the Pochhammer symbol defined for $\lambda\in\mathbb{C}$ and $n\in\mathbb{N}$ as:
\begin{equation}\label{2}
(\lambda)_0=1, \qquad (\lambda)_n=\lambda(\lambda+1)(\lambda+2)\cdots(\lambda+n-1); n\in\mathbb{N}.
\end{equation}

The series defined in (\ref{1}) is absolutely convergent for $|z|<1$ and $|z|=1$ (see \cite{Erd1953}).
Saigo \cite{Saigo1978} introduced the following left and right sided generalized integral transforms defined for $x>0$ respectively as:

\begin{align}
&\left(I_{0+}^{\alpha, \beta, \eta}f\right)(x)=\frac{x^{-\alpha-\beta}}{\Gamma(\alpha)}\notag\\
&\times\int\limits_{0}^{x}(x-t)^{\alpha-1}
\quad_2F_1\left(
\begin{array}{c}
\alpha+\beta,-\eta;\alpha;1-\frac{t}{x}\\
\end{array}
\right)f(t)dx, \label{3}
\end{align}
and
\begin{align}\label{4}
&\left(I_{-}^{\alpha, \beta, \eta}f\right)(x)=\frac{1}{\Gamma(\alpha)}\notag\\
&\times\int\limits_{x}^{\infty}(x-t)^{\alpha-1}t^{-\alpha-\beta}
\quad_2F_1\left(
\begin{array}{c}
\alpha+\beta,-\eta;\alpha;1-\frac{x}{t} \\
\end{array}
\right)f(t)dx,
\end{align}
where $\alpha, \beta, \eta \in\mathbb{C}$ and $\mathbb{R}e(\alpha)>0$ and $_2F_1(a,b;c;z)$ is Gauss hypergeometric function defined in (\ref{1}). When $\beta=-\alpha$, then (\ref{2}) and (\ref{4}) will lead to the classical Riemann-Liouville left and right-sided fractional integrals of order $\alpha\in\mathbb{C}$, $\mathfrak{R}(\alpha)>0$, (see \cite{Samko1993}):
\begin{eqnarray}\label{5}
\left(
  \begin{array}{c}
    I_{0+}^{\alpha, \beta, \eta}f \\
  \end{array}
\right)(x)=\frac{x^{-\alpha-\beta}}{\Gamma(\alpha)}\int\limits_{0}^{x}(x-t)^{\alpha-1}
f(t)dx (x>0),
\end{eqnarray}
and
\begin{eqnarray}\label{6}
\left(
  \begin{array}{c}
    I_{0+}^{\alpha, \beta, \eta}f \\
  \end{array}
\right)(x)=\frac{1}{\Gamma(\alpha)}\int\limits_{x}^{\infty}(x-t)^{\alpha-1}t^{-\alpha-\beta}f(t)dx (x>0).
\end{eqnarray}
If $\beta=0$, then equations (\ref{3}) and (\ref{4}) will reduce to the well known Erd\'{e}lyi-Kober fractional defined as:
\begin{eqnarray}\label{7}
\left(
  \begin{array}{c}
    I_{0+}^{\alpha, 0, \eta}f \\
  \end{array}
\right)(x)=\left(
  \begin{array}{c}
    K_{\eta,\alpha}^{+}f \\
  \end{array}
\right)(x)=\frac{x^{-\alpha-\beta}}{\Gamma(\alpha)}\int\limits_{0}^{x}(x-t)^{\alpha-1}
t^\eta f(t)dx
\end{eqnarray}
and
\begin{eqnarray}\label{8}
\left(
  \begin{array}{c}
    I_{0+}^{\alpha, 0, \eta}f \\
  \end{array}
\right)(x)=\left(
  \begin{array}{c}
    K_{\eta, \alpha}^{-}f \\
  \end{array}
\right)(x)=\frac{x^\eta}{\Gamma(\alpha)}\int\limits_{x}^{\infty}(x-t)^{\alpha-1}t^{-\alpha-\eta}
f(t)dx,
\end{eqnarray}
where $\alpha, \eta\in \mathbb{C}$, $\mathfrak{R}(\alpha)>0$ (see \cite{Samko1993}).\\
The generalized $k$-Bessel function  defined in \cite{Saiful2016} as:
\begin{eqnarray}\label{9}
W_{v, c}^{k}(z)=\sum\limits_{n=0}^{\infty}\frac{(-c)^n}{\Gamma_k(nk+v+k)n!}(\frac{z}{2})^{2n+\frac{v}{k}},
\end{eqnarray}
where $k>0$, $v>-1$, and $c\in\mathbb{R}$ and $\Gamma_k(z)$ is the $k$-gamma function defined in \cite{Diaz2007} as:
 \begin{eqnarray}
 \Gamma_k(z)=\int\limits_{0}^{\infty}t^{z-1}e^{-\frac{t^k}{k}}dt, z\in\mathbb{C}.
 \end{eqnarray}
 By inspection the following relation holds:
 \begin{eqnarray}
 \Gamma_k(z+k)=z\Gamma_k(z)
 \end{eqnarray}
 and
 \begin{eqnarray}
 \Gamma_k(z)=k^{\frac{z}{k}-1}\Gamma(\frac{z}{k}).
 \end{eqnarray}
 If $k\rightarrow 1$ and $c=1$, then the generalized $k$-Bessel function defined in (\ref{9}) reduces to the well known classical Bessel function $J_v$ defined in \cite{Erdelyi1953}. For further detail about $k$-Bessel function and its properties (see \cite{Gehlot2014}-\cite{Gehlot2016}). \\
The generalized hypergeometric function $_pF_{q}(z)$ is defined in \cite{Erd1953} as:
$$_pF_{q}(z)=\quad_pF_{q,}
                                 \left[
                                   \begin{array}{ccc}
                                     (\alpha_1),(\alpha_2),\cdots(\alpha_p) &  &  \\
                                      &  & ;z \\
                                     (\beta_1),(\beta_2),\cdots(\beta_q) &  &  \\
                                   \end{array}
                                 \right]$$
\begin{eqnarray}\label{10}
=\sum\limits_{n=0}^{\infty}\frac{(\alpha_1)_{n}(\alpha_2)_{n}\cdots(\alpha_p)_{n}}
{(\beta_1)_{n}(\beta_2)_{n}\cdots(\beta_q)_{n}}\frac{z^{n}}{n!},
\end{eqnarray}
 where $\alpha_i, \beta_j\in\mathbb{C}$; $i=1,2,\cdots,p$, $j=1,2,\cdots,q$ and $b_j\neq0, -1, -2,\cdots$
 and $(z)_{n}$ is the Pochhammer symbols. The gamma function is defined  as:
 \begin{eqnarray}
 \Gamma(\mu)=\int\limits_{0}^{\infty}t^{\mu-1}e^{-t}dt, \mu\in\mathbb{C},
 \end{eqnarray}
 \begin{eqnarray}\label{29}
 \Gamma(z+n)=(z)_{n}\Gamma(z), z\in\mathbb{C},
\end{eqnarray}
and beta function is defined as:
\begin{eqnarray}\label{w}
B(x,y)=\int\limits_{0}^{1}t^{x-1}(1-t)^{y-1}dt.
\end{eqnarray}
Also, the following identity of Gauss hypergeometric function holds:
\begin{eqnarray}\label{30}
_2F_{1}(a, b;c;1)=\frac{\Gamma(c)\Gamma(c-a-b)}{\Gamma(c-a)\Gamma(c-b)}; \mathfrak{R}(c-a-b)>0,
\end{eqnarray}
(see \cite{Erd1953}, \cite{Samko1993}).\\
The Wright type hypergeometric function  is defined  (see \cite{Wright1935}-\cite{Wrigt1935}) by the following series as:
$$_p\Psi_{q}(z)=\quad_p\Psi_{q}
                                 \left[
                                   \begin{array}{ccc}
                                     (\alpha_i, A_i)_{1,p} &  &  \\
                                      &  & ;z \\
                                     (\beta_j, B_j)_{1,q} &  &  \\
                                   \end{array}
                                 \right]$$
\begin{eqnarray}\label{11}
&=&\sum\limits_{n=0}^{\infty}\frac{\Gamma(\alpha_{1}+A_{1}n)\cdots\Gamma(\alpha_{p}+A_{p}n)}{\Gamma(\beta_{1}+B_{1}n)\cdots
\Gamma(\beta_{q}+B_{q}n)}\frac{z^{n}}{n!}
\end{eqnarray}
where $\beta_{r}$ and $\mu_{s}$  are real positive numbers such that
\begin{eqnarray*}
1+\sum\limits_{s=1}^{q}\beta_{s}-\sum\limits_{r=1}^{p}\alpha_{r}>0.
\end{eqnarray*}

Equation (\ref{11}) differs from the generalized hypergeometric function $_{p}F_{q}(z)$ defined (\ref{10})  only by a constant multiplier. The generalized hypergeometric function $_{p}F_{q}(z)$ is a special case of $_p\Psi_{q}(z)$ for $A_i=B_j=1$, where $i=1,2,\cdots,p$ and $j=1,2,\cdots,q$:
\begin{align}
\frac{1}{\prod\limits_{j=1}^{q}\Gamma(\beta_j)}\quad_pF_{q}
                                 \left[
                                   \begin{array}{ccc}
                                     (\alpha_1),\cdots(\alpha_p) &  &  \\
                                      &  & ;z \\
                                     (\beta_1),\cdots(\beta_q) &  &  \\
                                   \end{array}
                                 \right]=\frac{1}{\prod\limits_{i=1}^{p}\Gamma(\alpha_i)}\quad_p\Psi_{q}
                                 \left[
                                   \begin{array}{ccc}
                                     (\alpha_i, 1)_{1,p} &  &  \\
                                      &  & ;z \\
                                     (\beta_j, 1)_{1,q} &  &  \\
                                   \end{array}
                                 \right].\label{12}
\end{align}
For various properties of this function see \cite{Kilb2002}.\\
\textbf{Lemma 1.1} (A. A. Kilbas and N. Sebastian \cite{Kilbas2008}) Let $\alpha, \beta, \eta\in\mathbb{C}$, $\mathfrak{R}(\alpha)>0$ and $\mathfrak{\lambda}>\max[0, \mathfrak{\beta-n}]$, then the following relation holds:
\begin{eqnarray}\label{13}
\left(
  \begin{array}{c}
    I_{0+}^{\alpha, \beta, \eta}t^{\lambda-1} \\
  \end{array}
\right)(x)=\frac{\Gamma(\lambda)\Gamma(\lambda+\eta-\beta)}{\Gamma(\lambda-\beta)\Gamma(\lambda+\alpha+\eta)}x^{\lambda-\beta-1}
\end{eqnarray}
\textbf{Lemma 1.2} (A. A. Kilbas and N. Sebastian \cite{Kilbas2008}) Let $\alpha, \beta, \eta\in\mathbb{C}$, $\mathfrak{R}(\alpha)>0$ and $\mathfrak{\lambda}>\max[0, \mathfrak{\beta-n}]$, then the following relation holds:
\begin{eqnarray}\label{14}
\left(
  \begin{array}{c}
    I_{-}^{\alpha, \beta, \eta}t^{\lambda-1} \\
  \end{array}
\right)(x)=\frac{\Gamma(\eta-\lambda+1)\Gamma(\beta-\lambda+1)}{\Gamma(1-\lambda)\Gamma(\alpha+\beta+\eta-\lambda+1)}x^{\lambda-\beta-1}
\end{eqnarray}
In the same paper, they define the following left and right sided Erd\'{e}lyi-Kober fractional integral as:
\begin{eqnarray}\label{15}
\left(
  \begin{array}{c}
    K_{\eta,\alpha}^{+}t^{\lambda-1} \\
  \end{array}
\right)(x)=\frac{\Gamma(\lambda+n)}{\Gamma(\lambda+\alpha+\eta)}x^{\lambda-1},
\end{eqnarray}
where $\mathfrak{R}(\alpha)>0$, $\mathfrak{R}(\lambda)>-\mathfrak{R}(\eta)$, and
\begin{eqnarray}\label{16}
\left(
  \begin{array}{c}
    K_{\eta,\alpha}^{-}t^{\lambda-1} \\
  \end{array}
\right)(x)=\frac{\Gamma(\eta-\lambda+1)}{\Gamma(\alpha+\eta-\lambda+1)}x^{\lambda-1},
\end{eqnarray}
where  $\mathfrak{R}(\lambda)<1+\mathfrak{R}(\eta)$.\\


\section {Representation of Generalized fractional integrals in term of Wright functions}


In this section, we introduce the generalized left-sided fractional integration (\ref{3}) of the $k$-Bessel functions (\ref{9}). It is given by the following result.

\begin{theorem}\label{Theorem 2.1}
Assume that $\alpha$, $\beta$, $\eta$, $\lambda$, $v\in\mathbb{C}$  be such that
\begin{eqnarray}\label{34}
\mathfrak{R}(v)>-1, \mathfrak{R}(\alpha)>0, \mathfrak{R}(\lambda+v)>\max[0,\mathfrak{R}(\beta-\eta)],
\end{eqnarray}
then the following result holds:
\begin{align}
&\left(I_{0+}^{\alpha,\beta,\eta}t^{\frac{\lambda}{k}-1}W_{v,c}^{k}(t) \right)(x)=\frac{x^{\frac{\lambda}{k}+\frac{v}{k}-\beta-1}}{(2k)^{\frac{v}{k}}}\notag\\
&\times{}_2\Psi_{3}
\left[
\begin{array}{ccc}
 (\frac{\lambda}{k}+\frac{v}{k}, 2),(\frac{\lambda}{k}+\frac{v}{k}+\eta-\beta, 2k) \\
 &| -\frac{cx^2}{4k}\\
(\frac{\lambda}{k}+\frac{v}{k}-\beta,2), (\frac{\lambda}{k}+\frac{v}{k}+\alpha+\eta, 2),(\frac{v}{k}+1,1) \\
\end{array}
\right]\label{35}.
\end{align}
\end{theorem}

\begin{proof} 
Note that the condition
$$1+\sum\limits_{j=1}^{q}\beta_j- \sum\limits_{i=1}^{p}\alpha_i>0$$
is satisfied so therefore $_2\Psi_{3}(z)$ is defined. Now, from (\ref{3}) and (\ref{9}), we have
\begin{eqnarray*}
\left(
  \begin{array}{c}
    I_{0+}^{\alpha,\beta,\eta}t^{\frac{\lambda}{k}-1}W_{v,c}^{k}(t) \\
  \end{array}
\right)(x)=\sum\limits_{n=0}^{\infty}\frac{(-c)^n(\frac{1}{2})^{\frac{v}{k}+2n}}{\Gamma_k(v+k+nk)n!}
\left(
  \begin{array}{c}
    I_{0+,k}^{\alpha,\beta,\eta}t^{\frac{\lambda+v}{k}+2n-1} \\
  \end{array}
\right)(x)
\end{eqnarray*}
By equation (\ref{34}) and for any  $n=0,1,2,\cdots$, $\mathfrak{R}(\lambda+v+2nk)\geq\mathfrak{R}(\lambda+v)>\max[0,\mathfrak{R}(\beta-\eta)]$. Applying equation (\ref{14}), we obtain
\begin{align}
&\left(I_{0+,k}^{\alpha,\beta,\eta}t^{\frac{\lambda}{k}-1}W_{v,c}^{k}(t) \right)(x)=\frac{x^{\frac{\lambda+v}{k}-\beta-1}}{2^{\frac{v}{k}}}\notag\\
&\times\sum\limits_{n=0}^{\infty}\frac{\Gamma(\frac{v}{k}+\frac{\lambda}{k}+2n)\Gamma(\frac{v}{k}+\frac{\lambda}{k}+\eta-\beta+2n)}
{\Gamma(\frac{v}{k}+\frac{\lambda}{k}-\beta+2n)\Gamma(\frac{v}{k}+\frac{\lambda}{k}+\alpha+\eta+2n)\Gamma_k(\frac{v}{k}+1+n)
k^{\frac{v}{k}}}\notag\\
&\times\frac{(-cx^2)^n}{(4k)^n n!}\label{A}
\end{align}
By equation (\ref{11}), we obtain
\begin{align*}
&\left(I_{0+,k}^{\alpha,\beta,\eta}t^{\frac{\lambda}{k}-1}W_{v,c}^{k}(t) \right)(x)\\
&=\frac{x^{\frac{v}{k}+\frac{\lambda}{k}-\beta-1}}{(2k)^{\frac{v}{k}}}{}_2\Psi_{3}
\left[
\begin{array}{ccc}
 (\frac{v}{k}+\frac{\lambda}{k}, 2),(\frac{v}{k}+\frac{\lambda}{k}+\eta-\beta, 2) \\
 &| -\frac{cx^2}{4k}\\
(\frac{v}{k}+\frac{\lambda}{k}-\beta,2), (\frac{v}{k}+\frac{\lambda}{k}+\alpha+\eta, 2),(\frac{v}{k}+1,1) \\
\end{array}
\right].
\end{align*}
This is the required proof of (\ref{35}).
\end{proof}

\begin{corollary}\label{Corollary 2.2 } 
Assume that $\alpha$,  $\lambda$, $v\in\mathbb{C}$  be such that
$\mathfrak{R}(v)>-1$, $\mathfrak{R}(\alpha)>0$, $\mathfrak{R}(\lambda+v)>0$,
then the following result holds:
\begin{eqnarray*}
\left(
  \begin{array}{c}
    I_{0+}^{\alpha}t^{\frac{\lambda}{k}-1}W_{v,c}^{k}(t) \\
  \end{array}
\right)(x)&=&\frac{x^{\frac{v}{k}+\frac{\lambda}{k}+\alpha-1}}{(2k)^{\frac{v}{k}}}
\end{eqnarray*}
\begin{eqnarray}\label{36}
&\times&_1\Psi_{2}
\left[
\begin{array}{ccc}
 (v+\lambda, 2k) \\
 &   & | -\frac{cx^2}{4k}\\
 (\frac{v}{k}+\frac{\lambda}{k}+\alpha, 2),(\frac{v}{k}+1,k) \\
\end{array}
\right].
\end{eqnarray}
\end{corollary}

\begin{proof} 
By substituting $\beta=-\alpha$ in (\ref{35}), we obtain the required result.
\end{proof}

\begin{corollary}\label{Corollary 2.3} 
 Assume that $\alpha$, $\eta$, $\lambda$, $v\in\mathbb{C}$  be such that
$\mathfrak{R}(v)>-1$, $\mathfrak{R}(\alpha)>0$, $\mathfrak{R}(\lambda+v)>0$,
then the following formula holds:
\begin{eqnarray*}
\left(
  \begin{array}{c}
    K_{\alpha,\eta}^{+}t^{\frac{\lambda}{k}-1}W_{v,c}^{k}(t) \\
  \end{array}
\right)(x)&=&\frac{x^{\frac{v}{k}+\frac{\lambda}{k}-1}}{(2k)^{\frac{v}{k}}}
\end{eqnarray*}
\begin{eqnarray}\label{37}
&\times&_1\Psi_{2}
\left[
\begin{array}{ccc}
 (\frac{v}{k}+\frac{\lambda}{k}+\eta, 2) \\
 &| -\frac{cx^2}{4k}\\
 (\frac{v}{k}+\frac{\lambda}{k}+\alpha+\eta, 2),(\frac{v}{k}+1,1) \\
\end{array}
\right].
\end{eqnarray}
\end{corollary}

\begin{proof} 
By setting $\beta=0$ in (\ref{35}), we get the desired result.
\end{proof} 

\begin{theorem}\label{Theorem 2.4}
Assume that $\alpha$, $\beta$, $\eta$, $\lambda$, $v\in\mathbb{C}$ and $k>0$ be such that
\begin{eqnarray}\label{38}
\mathfrak{R}(v)>-1, \mathfrak{R}(\alpha)>0, \mathfrak{R}(\lambda-v)<1+\min[\mathfrak{R}(\beta),\mathfrak{R}(\eta)],
\end{eqnarray}
then the following result holds:
\begin{eqnarray*}
\left(
  \begin{array}{c}
    I_{0-}^{\alpha,\beta,\eta}t^{\frac{\lambda}{k}-1}W_{v,c}^{k}(\frac{1}{t}) \\
  \end{array}
\right)(x)&=&\frac{x^{\frac{\lambda-v}{k}-\beta-1}}{(2k)^{\frac{v}{k}}}
\end{eqnarray*}
\begin{multline}\label{39}
\times\quad_2\Psi_{3}
\left[
\begin{array}{ccc}
 (1+\beta-\frac{\lambda}{k}+\frac{v}{k}, 2),(1-\frac{\lambda}{k}+\frac{v}{k}+\eta, 2) \\
 &   & | -\frac{c}{4x^2}\\
(1-\frac{\lambda}{k}+\frac{v}{k},2), (1+\beta+\alpha+\eta-\frac{\lambda}{k}+\frac{v}{k}, 2),(\frac{v}{k}+1,k) \\
\end{array}
\right].
\end{multline}
\end{theorem}
\begin{proof}
Note that the condition
$$1+\sum\limits_{j=1}^{q}\beta_j- \sum\limits_{i=1}^{p}\alpha_i>0$$ $(i.e., 1>-1)$
is satisfied so therefore  $_2\Psi_{3}(z)$ is defined. Now,
from (\ref{4}) and (\ref{9}), we have
\begin{eqnarray*}
\left(
  \begin{array}{c}
    I_{0-}^{\alpha,\beta,\eta}t^{\frac{\lambda}{k}-1}W_{v,c}^{k}(\frac{1}{t}) \\
  \end{array}
\right)(x)=\sum\limits_{n=0}^{\infty}\frac{(-c)^n(\frac{1}{2})^{\frac{v}{k}+2n}}{\Gamma_k(v+k+nk)n!}
\left(
  \begin{array}{c}
    I_{0-}^{\alpha,\beta,\eta}t^{\frac{\lambda}{k}+\frac{v}{k}-2n-1} \\
  \end{array}
\right)(x)
\end{eqnarray*}
By equation (\ref{38}) and for any $k>0$ and $n=0,1,2,\cdots$, $\mathfrak{R}(\lambda-v-2n-1)\leq 1+\mathfrak{R}(\lambda-v-1)<1+\min[\mathfrak{\beta},\mathfrak{R}(\eta)]$. Applying equation (\ref{14}), we obtain
\begin{align}
&\left(I_{0-}^{\alpha,\beta,\eta}t^{\frac{\lambda}{k}-1}W_{v,c}^{k}(\frac{1}{t})\right)(x)=\frac{x^{\frac{\lambda}{k}+\frac{v}{k}-\beta-1}}{(2k)^{\frac{v}{k}}}\notag\\
&\times\sum\limits_{n=0}^{\infty}\frac{\Gamma(\beta-\frac{\lambda}{k}+\frac{v}{k}+1+2n)\Gamma(\eta-\frac{\lambda}{k}+\frac{v}{k}+1+2n)}
{\Gamma(1-\frac{\lambda}{k}+\frac{v}{k}+2n)\Gamma(\alpha+\beta+\eta-\frac{\lambda}{k}+\frac{v}{k}+1+2n)\Gamma(\frac{v}{k}+1+n)}\notag\\
&\times\frac{(-c)^n}{(4kx^2)^n n!}\label{B}
\end{align}
By equation (\ref{11}), we obtain
\begin{align*}
&\left( I_{0-}^{\alpha,\beta,\eta}t^{\frac{\lambda}{k}-1}W_{v,c}^{k}(\frac{1}{t})\right)(x)=\frac{x^{\frac{\frac{\lambda}{k}-\frac{v}{k}-\beta}{k}-1}}{(2k)^{\frac{v}{k}}}\\
&\times_2\Psi_{3}\left[
\begin{array}{ccc}
 (\beta-\frac{\lambda}{k}+\frac{v}{k}+1,2),(\eta-\frac{\lambda}{k}+\frac{v}{k}+1,2) \\
 & | -\frac{c}{4kx^2}\\
(1-\frac{\lambda}{k}+\frac{v}{k},2), (\alpha+\beta+\eta-\frac{\lambda}{k}+\frac{v}{k}+1,2),(\frac{v}{k}+1,1) \\
\end{array}
\right].
\end{align*}
This is the required proof of (\ref{39}).
\end{proof}

\begin{corollary}\label{Corollary 2.5} 
Assume that $\alpha$, $\eta$  $\lambda$, $v\in\mathbb{C}$ and $k>0$ be such that
$\mathfrak{R}(v)>-1$, $0<\mathfrak{R}(\alpha)<1-\mathfrak{R}(\lambda-v)$,
then the following result holds:
\begin{eqnarray*}
\left(
  \begin{array}{c}
    I_{0+}^{\alpha}t^{\frac{\lambda}{k}-1}W_{v,c}^{k}(\frac{1}{t}) \\
  \end{array}
\right)(x)&=&\frac{x^{\frac{\lambda}{k}-\frac{v}{k}+\alpha-1}}{(2k)^{\frac{v}{k}}}
\end{eqnarray*}
\begin{eqnarray}\label{40}
&\times&_1\Psi_{2}
\left[
\begin{array}{ccc}
 (1-\alpha-\frac{\lambda}{k}+\frac{v}{k}, 2) \\
 &| -\frac{c}{4kx^2}\\
 (1-\frac{\lambda}{k}+\frac{v}{k}, 2),(\frac{v}{k}+1,1) \\
\end{array}
\right].
\end{eqnarray}
\end{corollary}

\begin{corollary}\label{Corollary 2.6} 
Assume that $\alpha$, $\eta$, $\lambda$, $v\in\mathbb{C}$ and $k>0$ be such that
$\mathfrak{R}(v)>-1$, $\mathfrak{R}(\alpha)>0$, $\mathfrak{R}(\lambda+v)<1+\max[0,\mathfrak{R}(\eta)]$,
then the following formula holds:
\begin{eqnarray*}
\left(
  \begin{array}{c}
    K_{\alpha,\eta}^{-}t^{\frac{\lambda}{k}-1}W_{v,c}^{k}(\frac{1}{t}) \\
  \end{array}
\right)(x)&=&\frac{x^{\frac{\lambda}{k}-\frac{v}{k}-1}}{(2k)^{\frac{v}{k}}}
\end{eqnarray*}
\begin{eqnarray}\label{41}
&\times&_1\Psi_{2}
\left[
\begin{array}{ccc}
 (1+-\frac{\lambda}{k}+\frac{v}{k}+\eta, 2) \\
 &   & | -\frac{c}{4kx^2}\\
 (1-\frac{\lambda}{k}+\frac{v}{k}+\alpha+\eta, 2),(\frac{v}{k}+1,1) \\
\end{array}
\right].
\end{eqnarray}
\end{corollary}
\section{Representation in terms of generalized hypergeometric function}

In this section, we introduce the generalized fractional integrals of $k$-Bessel function in term of generalized hypergeometric function. First we consider the following well known results:
\begin{eqnarray}\label{42}
\Gamma(2z)=\frac{2^{2z-1}}{\sqrt\pi}\Gamma(z)\Gamma(z+\frac{1}{2}); z\in\mathbb{C}
\end{eqnarray}
and
\begin{eqnarray}\label{43}
(z)_{2n}=2^{2n}(\frac{z}{2})_n(\frac{z+1}{2})_n, z\in\mathbb{C}, n\in\mathbb{N}.
\end{eqnarray}
We represent the following theorems containing the generalized hypergeometric function.

\begin{theorem}\label{Theorem 3.1} 
Assume that $\alpha$, $\beta$, $\eta$, $\lambda$, $v\in\mathbb{C}$  be such that
\begin{eqnarray}\label{44}
\mathfrak{R}(v)>-1, \mathfrak{R}(\alpha)>0, \mathfrak{R}(\lambda+v)>\max[0,\mathfrak{R}(\beta-\eta)],
\end{eqnarray}
and let $\frac{\lambda}{k}+\frac{v}{k}$, $\frac{\lambda}{k}+\frac{v}{k}+\eta-\beta\neq0,-1,\cdots$, then the following result holds:
\begin{align}
&\left(I_{0+}^{\alpha,\beta,\eta}t^{\frac{\lambda}{k}-1}W_{v,c}^{k}(t)\right)(x)=\frac{x^{\frac{\lambda}{k}+\frac{v}{k}-\beta-1}}{(2k)^{\frac{v}{k}}}\frac{\Gamma(\frac{\lambda}{k}+\frac{v}{k})\Gamma(\frac{\lambda}{k}+\frac{v}{k}+\eta-\beta)}
{\Gamma(\frac{\lambda}{k}+\frac{v}{k}-\beta)\Gamma(\frac{\lambda}{k}+\frac{v}{k}+\alpha+\eta)\Gamma(\frac{v}{k}+1)}\notag\\
&\times_4F_{5}\left[
\begin{array}{ccc}
 \frac{\lambda}{2k}+\frac{v}{2k}, \frac{\lambda}{2k}+\frac{v}{2k}+\frac{1}{2}, \frac{\lambda}{2k}+\frac{v}{2k}+\frac{\eta-\beta}{2}, \frac{\lambda}{2k}+\frac{v}{2k}+\frac{\eta-\beta+1}{2} \\
 &|-\frac{cx^2}{4k}\\
\frac{v}{k}+1,\frac{\lambda}{2k}+\frac{v}{2k}-\frac{\beta}{2},\frac{\lambda}{2k}+\frac{v}{2k}-\frac{\beta+1}{2},\frac{\lambda}{k}+\frac{v}{k}+
\frac{\alpha+\eta}{2}, \frac{\lambda}{2k}+\frac{v}{2k}+\frac{\alpha+\eta+1}{2} \label{45}\\
\end{array}
\right].
\end{align}
\end{theorem}

\begin{proof}
Note that $_4F_5$ defined in (\ref{45}) exit as the series is absolutely convergent.
 Now, using (\ref{29}) with $z=\frac{v}{k}+1$ and (\ref{A}) and applying (\ref{43}) with $z$ being replaced by $\frac{\lambda}{k}+\frac{v}{k}$, $\frac{\lambda}{k}+\frac{v}{k}+\eta-\beta$ and $\frac{\lambda}{k}+\frac{v}{k}+\alpha+\eta$, we have
\begin{align*}
&\left(I_{0+}^{\alpha,\beta,\eta}t^{\frac{\lambda}{k}-1}W_{v,c}^{k}(t)\right)(x)=\frac{x^{\frac{\lambda+v}{k}-\beta-1}}{(2k)^{\frac{v}{k}}}\\
&\times\sum\limits_{n=0}^{\infty}\frac{\Gamma(\frac{v}{k}+\frac{\lambda}{k})\Gamma(\frac{v}{k}+\frac{\lambda}{k}+\eta-\beta)}
{\Gamma(\frac{v}{k}+\frac{\lambda}{k}-\beta)\Gamma(\frac{v}{k}+\frac{\lambda}{k}+\alpha+\eta)\Gamma(\frac{v}{k}+1)}\\
&\times\frac{(\frac{v}{k}+\frac{\lambda}{k})_{2n}(\frac{v}{k}+\frac{\lambda}{k}+\eta-\beta)_{2n}}{(\frac{v}{k}+\frac{\lambda}{k}-\beta)_{2n}(\frac{v}{k}+\frac{\lambda}{k}+\alpha+\beta)_{2n}}\frac{(-cx^2)^n}{(4k)^n n!}\\
&=\frac{x^{\frac{\lambda+v}{k}-\beta-1}}{(2k)^{\frac{v}{k}}}\frac{\Gamma(\frac{v}{k}+\frac{\lambda}{k})\Gamma(\frac{v}{k}+\frac{\lambda}{k}+\eta-\beta)}{\Gamma(\frac{v}{k}+\frac{\lambda}{k}-\beta)\Gamma(\frac{v}{k}+\frac{\lambda}{k}+\alpha+\eta)\Gamma(\frac{v}{k}+1)}\\
&\times\sum\limits_{n=0}^{\infty}\frac{(\frac{v}{2k}+\frac{\lambda}{2k})_{n}(\frac{v}{2k}+\frac{\lambda}{2k}+\frac{1}{2})_{n}
(\frac{v}{2k}+\frac{\lambda}{2k}+\frac{\eta-\beta}{2})_{n}(\frac{v}{2k}+\frac{\lambda}{2k}+\frac{\eta-\beta}{2})_{n}}
{(\frac{v}{k}+1)(\frac{v}{2k}+\frac{\lambda}{2k}-\frac{\beta}{2})_{n}(\frac{v}{2k}+\frac{\lambda}{2k}-\frac{\beta}{2})_{n}
(\frac{v}{2k}+\frac{\lambda}{2k}+\frac{\alpha+\eta}{2})_{n}(\frac{v}{2k}+\frac{\lambda}{2k}+\frac{\alpha+\eta+1}{2})_{n}}\\
&\times\frac{(-cx^2)^n}{(4k)^n n!}.
\end{align*}
Thus, in accordance with equation (\ref{10}), we get the required result (\ref{45}).
\end{proof}

\begin{corollary}\label{Corollary 3.2 }
Assume that $\alpha$,  $\lambda$, $v\in\mathbb{C}$  be such that
$\mathfrak{R}(v)>-1$, $\mathfrak{R}(\alpha)>0$, $\mathfrak{R}(\lambda+v)>0$ and $\frac{\lambda}{k}+\frac{v}{k}=0,-1,\cdots$,
then the following result holds:
\begin{align}
&\left(I_{0+}^{\alpha,\beta,\eta}t^{\frac{\lambda}{k}-1}W_{v,c}^{k}(t)\right)(x)=\frac{x^{\frac{\lambda}{k}+\frac{v}{k}+\alpha-1}}{(2k)^{\frac{v}{k}}}\frac{\Gamma(\frac{\lambda}{k}+\frac{v}{k})}{\Gamma(\frac{\lambda}{k}+\frac{v}{k}-\beta)\Gamma(\frac{v}{k}+1)}\notag\\
&\times_2F_{3}
\left[
\begin{array}{ccc}
 \frac{\lambda}{2k}+\frac{v}{2k}, \frac{\lambda}{2k}+\frac{v}{2k}+\frac{1}{2}, \\
 &| -\frac{cx^2}{4k}\\
\frac{v}{k}+1,\frac{\lambda}{2k}+\frac{v}{2k}-\frac{\beta}{2},\frac{\lambda}{2k}+\frac{v}{2k}-\frac{\beta+1}{2}
\end{array}
\right].
\end{align}
\end{corollary}

\begin{proof}
By substituting $\beta=-\alpha$ in (\ref{45}), we obtain the required result.
\end{proof}

\begin{corollary}\label{Corollary 3.3}
Assume that $\alpha$, $\eta$, $\lambda$, $v\in\mathbb{C}$  be such that
$\mathfrak{R}(v)>-1$, $\mathfrak{R}(\alpha)>0$, $\mathfrak{R}(\lambda+v)>0$ and let
$\frac{\lambda}{k}+\frac{v}{k}+\eta-\beta\neq0,-1,\cdots$, then the following result holds:
\begin{align}
&\left(K_{\alpha,\eta}^{+}t^{\frac{\lambda}{k}-1}W_{v,c}^{k}(t)\right)(x)=\frac{x^{\frac{\lambda}{k}+\frac{v}{k}-1}}{(2k)^{\frac{v}{k}}}
\frac{\Gamma(\frac{\lambda}{k}+\frac{v}{k}+\eta)}{\Gamma(\frac{\lambda}{k}+\frac{v}{k}+\alpha+\eta)\Gamma(\frac{v}{k}+1)}\notag\\
&\times{}_2F_{3}
\left[
\begin{array}{ccc}
   \frac{\lambda}{2k}+\frac{v}{2k}+\frac{\eta}{2}, \frac{\lambda}{2k}+\frac{v}{2k}+\frac{\eta+1}{2} \\
 &| -\frac{cx^2}{4k}\\
\frac{v}{k}+1,\frac{\lambda}{k}+\frac{v}{k}+
\frac{\alpha+\eta}{2}, \frac{\lambda}{2k}+\frac{v}{2k}+\frac{\alpha+\eta+1}{2} \\
\end{array}
\right].
\end{align}
\end{corollary}

\begin{proof}
By setting $\beta=0$ in (\ref{45}), we get the desired result.
\end{proof}

\begin{theorem}\label{Theorem 3.4} 
Assume that $\alpha$, $\beta$, $\eta$, $\lambda$, $v\in\mathbb{C}$ and $k>0$ be such that
\begin{eqnarray}\label{46}
\mathfrak{R}(v)>-1, \mathfrak{R}(\alpha)>0, \mathfrak{R}(\lambda-v)<1+\min[\mathfrak{R}(\beta),\mathfrak{R}(\eta)],
\end{eqnarray}
and let $\frac{\beta-\lambda}{k}+\frac{v}{k}+1$, $\eta-\frac{\lambda}{k}+\frac{v}{k}+1\neq0,-1,\cdots$, then the following result holds:
\begin{align}
&\left(I_{0-}^{\alpha,\beta,\eta}t^{\frac{\lambda}{k}-1}W_{v,c}^{k}(\frac{1}{t})\right)(x)=\frac{x^{\frac{\lambda}{k}-\frac{v}{k}-\beta-1}}{(2k)^{\frac{v}{k}}}\notag\\
&\times\frac{\Gamma(\beta-\frac{\lambda}{k}+\frac{v}{k}+1)\Gamma(\eta-\frac{\lambda}{k}+\frac{v}{k}+1)}{\Gamma(1-\frac{\lambda}{k}+\frac{v}{k})\Gamma(\alpha+\beta+\eta-\frac{\lambda}{k}+\frac{v}{k}+1)\Gamma(\frac{v}{k}+1)}\notag\\
&\times_4F_{5}
\left[
\begin{array}{ccc}
\frac{\beta+1}{2}-\frac{\lambda}{2k}+\frac{v}{2k}, \frac{\beta+2}{2}-\frac{\lambda}{2k}+\frac{v}{2k}, \frac{\eta+1}{2}-\frac{\lambda}{2k}+\frac{v}{2k}, \frac{\eta+2}{2}-\frac{\lambda}{2k}+\frac{v}{2k} \\
 &|-\frac{c}{4kx^2}\\
\frac{v}{k}+1,\frac{1}{2}-\frac{\lambda}{2k}+\frac{v}{2k},1-\frac{\lambda}{2k}+\frac{v}{2k},
\frac{\alpha+\beta+\eta+1}{2}-\frac{\lambda}{k}+\frac{v}{k},\frac{\alpha+\eta+2}{2}-\frac{\lambda}{2k}+\frac{v}{2k} \\
\end{array}
\right].\label{47}
\end{align}
\end{theorem}

\begin{proof}
Using (\ref{29}) with $z=\frac{v}{k}+1$ and (\ref{B}) and applying (\ref{43}) with $z$ being replaced by $\beta-\frac{\lambda}{k}+\frac{v}{k}+1$, $1-\frac{\lambda}{k}+\frac{v}{k}$ and $\beta-\frac{\lambda}{k}+\frac{v}{k}+\alpha+\eta+1$, we have
\begin{align*}
&\left(I_{0-}^{\alpha,\beta,\eta}t^{\frac{\lambda}{k}-1}W_{v,c}^{k}(\frac{1}{t})\right)(x)=\frac{x^{\frac{\lambda}{k}-\frac{v}{k}-\beta-1}}{(2k)^{\frac{v}{k}}}\\
&\frac{\Gamma(\beta-\frac{\lambda}{k}+\frac{v}{k}+1)\Gamma(\eta-\frac{\lambda}{k}+\frac{v}{k}+1)}
{\Gamma(1-\frac{\lambda}{k}+\frac{v}{k})\Gamma(\alpha+\beta+\eta-\frac{\lambda}{k}+\frac{v}{k}+1)\Gamma(\frac{v}{k}+1)}\\
&\times\sum\limits_{n=0}^{\infty}\frac{(\frac{\beta+1}{2}-\frac{\lambda}{2k}+\frac{v}{2k})_n
(\frac{\beta}{2}-\frac{\lambda}{2k}+\frac{v}{2k}+1)_n(\frac{\eta+1}{2}-\frac{\lambda}{2k}+\frac{v}{2k})_n
(\frac{\eta}{2}-\frac{\lambda}{2k}+\frac{v}{2k}+1)_n}{(\frac{v}{k}+1)_n((\frac{1}{2}-\frac{\lambda}{2k}+\frac{v}{2k})_n)
(1-\frac{\lambda}{2k}+\frac{v}{2k})_n(\frac{\alpha+\beta+\eta+1}{2}-\frac{\lambda}{2k}+\frac{v}{2k})_n
(\frac{\alpha+\beta+\eta}{2}-\frac{\lambda}{2k}+\frac{v}{2k}+1)_n}\\
&\times\frac{(-c)^n}{(4kx^2)^n n!}.
\end{align*}
By equation (\ref{10}), we obtain the required given in (\ref{47}).
\end{proof}

\begin{corollary}\label{Corollary 3.5 } 
Assume that $\alpha$, $\eta$  $\lambda$, $v\in\mathbb{C}$ and $k>0$ be such that
$\mathfrak{R}(v)>-1$, $0<\mathfrak{R}(\alpha)<1-\mathfrak{R}(\lambda-v)$, and let $\frac{\lambda}{k}-\frac{v}{k}+\alpha\neq1,2,\cdots$
then the following result holds:
\begin{eqnarray*}
\left(
  \begin{array}{c}
    I_{0+}^{\alpha}t^{\frac{\lambda}{k}-1}W_{v,c}^{k}(\frac{1}{t}) \\
  \end{array}
\right)(x)&=&\frac{x^{\frac{\lambda}{k}-\frac{v}{k}+\alpha-1}}{(2k)^{\frac{v}{k}}}
\frac{\Gamma(-\alpha-\frac{\lambda}{k}+\frac{v}{k}+1)}
{\Gamma(1-\frac{\lambda}{k}+\frac{v}{k})\Gamma(\frac{v}{k}+1)}
\end{eqnarray*}
\begin{multline}
\times{}_2F_{3}
\left[
\begin{array}{ccc}
 \frac{-\beta+1}{2}-\frac{\lambda}{2k}+\frac{v}{2k}, \frac{-\alpha+2}{2}-\frac{\lambda}{2k}+\frac{v}{2k}, \\
 & | -\frac{c}{4kx^2}\\
\frac{v}{k}+1,\frac{1}{2}-\frac{\lambda}{2k}+\frac{v}{2k},1-\frac{\lambda}{2k}+\frac{v}{2k},\\
\end{array}
\right].
\end{multline}
\end{corollary}

\begin{corollary}\label{Corollary 3.6} 
Assume that $\alpha$, $\eta$, $\lambda$, $v\in\mathbb{C}$ and $k>0$ be such that
$\mathfrak{R}(v)>-1$, $\mathfrak{R}(\alpha)>0$, $\mathfrak{R}(\lambda+v)<1+\max[0,\mathfrak{R}(\eta)]$ and let $\frac{\lambda}{k}-\frac{v}{k}-\eta\neq1,2,\cdots$,
then the following formula holds:
\begin{eqnarray*}
\left(
  \begin{array}{c}
    K_{\eta,\alpha}^{-}t^{\frac{\lambda}{k}-1}W_{v,c}^{k}(\frac{1}{t}) \\
  \end{array}
\right)(x)&=&\frac{x^{\frac{\lambda}{k}-\frac{v}{k}-1}}{(2k)^{\frac{v}{k}}}
\frac{\Gamma(\eta-\frac{\lambda}{k}+\frac{v}{k}+1)}
{\Gamma(1-\frac{\lambda}{k}+\frac{v}{k})\Gamma(\frac{v}{k}+1)}
\end{eqnarray*}
\begin{multline}
\times{}_2F_{3}
\left[
\begin{array}{ccc}
  \frac{\eta+1}{2}-\frac{\lambda}{2k}+\frac{v}{2k}, \frac{\eta+2}{2}-\frac{\lambda}{2k}+\frac{v}{2k} \\
 &   & | -\frac{c}{4kx^2}\\
\frac{v}{k}+1,\frac{\alpha+\eta+1}{2}-\frac{\lambda}{k}+\frac{v}{k},\frac{\alpha+\eta+2}{2}-\frac{\lambda}{2k}+\frac{v}{2k} \\
\end{array}
\right].
\end{multline}
\end{corollary}
 Corollary 3.5 and 3.6 follow from theorem 3.4 in respective cases $\beta=-\alpha$ and $\beta=0$.\\\\


\end{document}